\documentclass[12pt,a4paper,reqno]{amsart}
\usepackage{amsfonts}
\usepackage{amsthm}
\usepackage{amsmath}
\usepackage{times}
\usepackage{amscd}
\usepackage[latin2]{inputenc}
\usepackage{t1enc}
\usepackage{enumitem}
\usepackage[mathscr]{eucal}
\usepackage{indentfirst}
\usepackage{framed}
\usepackage{graphicx}
\usepackage{graphics}
\usepackage{pict2e}
\usepackage{epic}
\numberwithin{equation}{section}

\usepackage[margin=2.9cm]{geometry}
\usepackage{epstopdf} 
\usepackage{verbatim}
\usepackage{amssymb}
\usepackage{xcolor}
\usepackage{physics}
\usepackage{hyperref}

\theoremstyle{plain}
\newtheorem{theo}{Theorem}[section]
\newtheorem{lem}{Lemma}[section]

\newtheorem{prop}{Proposition}[section]

\theoremstyle{definition}
\newtheorem{defin}{Definition}[section]

\newtheorem{rem}{Remark}[section]
\newtheorem{ex}{Example}[section]

\newcommand{\gld}{\mathfrak{gl}(\mathfrak{d})}
\newcommand{\mfd}{\mathfrak{d}}
\newcommand{\mfs}{\mathfrak{s}}
\newcommand{\mfL}{\mathfrak{L}}
\newcommand{\sld}{\mathfrak{sl}(\mathfrak{d})}
\newcommand{\mfT}{\mathfrak{T}}
\newcommand{\der}{\partial}
\newcommand{\C}{\mathbb{C}}
\newcommand{\nicesum}[1]{\displaystyle \sum_{#1} }

\begin{document}
	
	\title[A bound on the degree of singular vectors for $E(5,10)$]{A bound on the degree of singular vectors for the exceptional Lie superalgebra $E(5,10)$}

	\author{Daniele Brilli}
	\address{Dipartimento di Matematica "Guido Castelnuovo", Universita' di Roma "La Sapienza", Piazzale Aldo Moro 5, Roma, Italia}
	
	\begin{abstract}
		We use the language of Lie pseudoalgebras to gain information about the representation theory of the simple infinite-dimensional linearly compact Lie superalgebra of exceptional type $E(5,10)$. This technology allows us to prove that the degree of singular vectors in minimal Verma modules is $\leq 14$. A few technical adjustments allow us to refine the bound, proving that the degree must always be $\leq 12$ and it is actually, except for a finite number of cases, $\leq 10$.
	\end{abstract}
	
	\date{June 22, 2020}
	
	\maketitle
	\section{Introduction}
	Infinite dimensional linearly compact simple Lie superalgebras over $\C$ were classified by Kac in \cite{K}. Besides Lie algebras in Cartan's list \cite{C}, the complete list consists of ten families and five exceptions, denoted by $E(1,6)$, $E(2,2)$, $E(3,6)$, $E(3,8)$, $E(4,4)$ and $E(5,10)$.\\
	In \cite{KR1} Kac and Rudakov started the study of representations of these exceptional superalgebras following the approach developed by Rudakov for the "non super" case, establishing the language of \emph{generalized Verma modules} and reducing the problem to the description of the so-called \emph{degenerate} modules and the study of \emph{singular vectors}.\\
	They completed the classification of degenerate Verma modules and singular vectors for $E(3,6)$ in \cite{KR1,KR2,KR3}. In the meanwhile they started to investigate the cases of $E(3,8)$ and $E(5,10)$ in \cite{KR4}. For $E(3,8)$ they could apply most of the arguments from \cite{KR2} and found all the degenerate Verma modules; they also described many degenerate $E(5,10)$-modules and conjectured there were no others.\\
	Afterwards Rudakov in \cite{R} related the problem with the study of morphism between Verma modules. He defined a degree of such a morphism and classified all the morphisms of degree $1$ (which correspond to the degenerate modules found in \cite{KR4}) but he also obtained morphisms of degree $2$ and $3$ as products of morphisms of degree $1$ and found morphisms of degree $4$ and $5$. He then conjectured there were no morphisms of higher degree and that his list was exhaustive.\\
	Cantarini and Caselli in \cite{CC} developed some combinatorial aspects of morphisms between Verma modules for $E(5,10)$ that allowed them, in particular, to confirm part of Rudakov's conjecture. Indeed, they showed that the morphisms found by Rudakov up to degree $3$ were the only ones.\\
	
	We briefly recall the definitions of generalized Verma module and singular vector.\\
	
	$E(5,10)=\oplus_{i\geq -2} \mfL_i$ is equipped with a $\mathbb{Z}$-grading of depth 2 consistent with the Lie superalgebra structure. In particular, $\mfL_0$ is a Lie algebra isomorphic to $\mathfrak{sl}_5$. If we take an ${sl}_5$-module $V$ and allow $\mfL_+=\oplus_{i>0} \mfL_i$ to act trivially on it, we obtain a module over $\mfL_{\geq 0}=\oplus_{i\geq0} \mfL_i$. We can then consider the induced $\mfL$-module
	\[
	\mfT(V)=U(\mfL)\otimes_{U(\mfL_{\geq 0})} V
	\]
	where $U(\mfL)$ is the universal enveloping algebra of $\mfL$.\\
	
	$\mfT(V)$ is called \emph{generalized Verma module}. One says it is \emph{minimal} if $V$ is an irreducible $\mathfrak{sl}_5$-module. If $\mfT(V)$ is minimal but not irreducible, it is said to be \emph{degenerate}.\\
	It is possible to define a grading on $\mfT(V)=\oplus_{p\geq 0} \mfT^p(V)$ compatible with that of $\mfL$, meaning that $\mfL_i \cdot \mfT^p(V)\subset \mfT^{p-1} (V)$.\\
	A \emph{singular vector} is an element of $\mfT(V)$ that is killed by $\mfL_+$.\\
	The notions of degeneracy of a minimal Verma module, existence of (non constant) singular vectors and (positive degree) morphisms are all equivalent (see, for instance, \cite[Prop. 3.5]{CC}).\\
	Basically, if one has a positive degree singular vector in a minimal Verma module $\mfT(V)$, the $\mfL$-submodule it generates is a proper submodule, thus $\mfT(V)$ is degenerate. On the other way, a $\mfL$-morphism of positive degree $p$ maps constant vectors, which are automatically singular, to singular vectors of positive degree $p$.\\
	
	Despite a visible fair amount of understanding of these objects, an explicit bound on the degree of singular vectors (or equivalently morphisms) in the literature is only implicitly conjectured. This article aims to fill this gap and provide such a bound, with the hope that the techniques developed will allow to improve the result and help "attack" these conjectures from "above" as well as from "below".\\
	
	The idea is the following: the even part of $E(5,10)$ is $S(5)$, the Lie algebra of zero-divergence vector fields in five indeterminates. This algebra is isomorphic to the \emph{annihilation algebra} of a \emph{Lie pseudoalgebra}. Lie pseudoalgebras are Lie algebras in a \emph{pseudotensor category}, hence their name (see \cite{BD,BDK1}). The analogs of Verma modules are called \emph{tensor modules} in this language. A structural correspondence between a Lie pseudoalgebra and its annihilation algebra guarantees a close bond between their representation theories (see \cite[Section 13]{BDK1}).
	(Finite) pseudoalgebras have the upside of having a developed theory (\cite{BDK1,BDK2,BDK3,BDK4,D}) and allow one to talk about singular vectors and their degrees in quite a "manageable" way. For istance, we know that the degree of singular vectors in tensor modules for Lie pseudoalgebras of type S is at most $2$, (Theorem \ref{theorem degrees tensor modules}, see \cite[Section 7]{BDK2} for more details).\\
	The fact that the pseudoalgebraic structure associated with the even part of $E(5,10)$ has such an immediate bound, together with the finiteness of the odd degree, prevent vectors with high enough degree from being singular.  \\
	
	The paper is organized as follows: in Section \ref{sectionPreliminariesE(5,10)} we recall the basic definitions about $E(5,10)$, generalized Verma modules and singular vectors and set up the notation. In Section \ref{sectionPreliminariesLiepseudoalgebras} we recall the basic notions about Lie pseudoalgebras, annihilation algebras and the correspondence between their representation theory. Section \ref{sectionPrimitiveLiepseudoalgebras} is dedicated to a brief summary of results about representation theory of primitive Lie pseudoalgebras of type $W$ and $S$. In Section \ref{sectionBoundSingularvectors} we build a finite filtration of $S(5)$-submodules on a generalized Verma module and realize their quotients as tensor modules, which allows us to apply the results stated in the previous section and prove our first result. Finally, in Section \ref{sectionBoundRefining} we use a very easy technical lemma to refine the bound.

	\section{E(5,10), Verma modules and notation}\label{sectionPreliminariesE(5,10)}
	We deal with the simple linearly compact Lie superalgebra of exceptionaly type $\mfL =E(5,10)$ and we will use the geometric construction provided in \cite[5.3]{CK}\\
	Let $\mfd = (\C^5)^\ast$ and let $\{ \der_1 ,\dotsc, \der_5\}$ and $\{x_1, \dotsc, x_5\}$ be bases for respectively $\mfd$ and $\mfd^\ast$.\\
	We can realize the even part of $\mfL$ as zero-divergence vector fields in the indeterminates $x_1, \dotsc, x_5$,
	\begin{equation*}
	\mfL_{(0)} = S(5) = \Big\{ D = \nicesum{i=1}^{5} f_i \der_i \, | \, f_i\in C[[x_1, \dotsc, x_5]], \, div(D)=0\Big\},\\
	\end{equation*}
	and the odd part as closed 2-forms in the same indeterminates
	\begin{align*}
	\mfL_{(1)} = d\Omega^1(5) = \Big\{& \omega = \nicesum{i<j=1}^5 f_{ij}\xi_{ij}\, | \, f_{ij}\in \C[[x_1,\dotsc, x_5]], \, \xi_{ij}= \dd x_i\wedge \dd x_j ,\\
	 &\omega = \dd\alpha \text{ for some } \alpha\in \Omega^1(5)\Big\}.
	\end{align*}
	The bracket between even elements is the one of the Lie algebra $S(5)$, while the brackets between even and odds elements are given by the Lie derivative
	\begin{equation*}
	[D,\omega] = \mathcal{L}_D(\omega) = \dd(\iota_D(\omega)),
	\end{equation*}
	where $\iota_D(\omega)$ is the contraction of $\omega$ by $D$.\\
	Finally, the bracket between odd elements is given by wedge product and consequent identification with a zero-divergence vector field:
	\begin{equation*}
	[\omega_1,\,\omega_2] = D \qq{where D is s.t.} \iota_D(v) = \omega_1\wedge\omega_2,
	\end{equation*}
	where $v$ is a fixed volume form. $\dd \omega_1 \wedge \dd\omega_2$ is a closed form and it is easy to check that under this identification it corresponds to a zero-divergence vector field.\\
	
	In order to give an explicit formula, for $i,j,h,k\in\{1,\dotsc, 5\}$ we set
	\[
	(ijhk)= \begin{cases}
	0 \quad\text{ if } |\{i,j,h,k\}|<4 \\
	l \quad\text{ otherwise, where } \{i,j,h,k, l\} = \{1,\dotsc, 5\};
	\end{cases}
	\]
	we shall adopt the following convention: whenever such an index occurs in an expression and it takes on the value 0, so does the expression.\\
	
	This way the bracket between two odd elements can be defined as
	\begin{equation}\label{bracketodd}
	[\xi_{ij},\xi_{hk}] = \varepsilon_{(ijhk)}\der_{(ijhk)}
	\end{equation}
	and then extended by $\C[[x_1,\dotsc,x_5]]$-bilinearity, where  $\varepsilon_{(ijhk)}$ is the sign of the permutation $(ijhkl)$ if $(ijhk)=l\neq 0$, $0$ otherwise; in a similar way $\der_{(ijhk)}=\der_l$ if $(ijhk)=l\neq0$, $\der_{(ijhk)}=0$ otherwise.\\
	
	Setting $deg\, x_i= -deg\, \der_i = 2$ and $deg\, \xi_{hk}=-1$ provides $\mfL=\bigoplus_{i\geq -2}\mfL_{i}$ with a transitive, irreducible $\mathbb{Z}$-grading of depth $2$ consistent with the superalgebra structure, which means that
	\begin{equation}
	[\mfL_n,\mfL_m]\subseteq \mfL_{n+m}
	\end{equation}
	We'll call $\mfL_-=\mfL_{-2}\oplus\mfL_{-1}$, $\mfL_{+}=\bigoplus_{i>0} \mfL_i$ and $\mfL_{\geq 0} = \mfL_{0}\oplus \mfL_{+}$.\\
	We have an isomorphism between $\mfL_0$ and $\sld$ given by
	\begin{equation}
	 x_i\der_j\longmapsto - e^i_j.
	\end{equation}
	Given the fact that $[\mfL_0,\mfL_n]\subseteq \mfL_n$, we can view $\mfL_n$ as an $\sld$-module; in particular, $\mfL_{-2}\cong \mfd$ and $\mfL_{-1} \cong \bigwedge^2 \mfd^\ast =: \mfs$.\\
	It is useful to describe also $\mfL_1$ as a $\sld$-module: it is the highest weight representation in $\mfd^\ast \otimes \bigwedge^2 \mfd^\ast$, (see \cite[Section 4.3]{CK}) and it is generated by the highest weight vector $x_1 \xi_{12}$.\\
	Notice that $\mfL_j = \mfL_1^j$ for $j\geq 1$ and that $\mfL_-\cong \mfd \oplus \mfs$ is a finite dimensional Lie superalgebra whose superbracket is non trivial only when restricted to the odd part, where is given by \eqref{bracketodd}.\\
	
	This grading extends to the universal enveloping algebra $U(\mfL)$, and in particular to $U(\mfL_-)$. In the latter case, as common practice, the sign of the degree is inverted in order to have a grading over $\mathbb{N}$.\\
	
	We will use this grading to study generalized Verma modules, which we will introduce here following \cite{R}(see also \cite{KR1,KR2,CC}).\\
	
	Given a $\sld\cong\mfL_0$-module $V$, we can extend it to a $\mfL_{\geq 0}$-module by letting $\mfL_{+}$ act trivially on it; we can then consider the induced $\mfL$-module
	\[
	\mfT(V)= U(\mfL)\otimes_{U(\mfL_{\geq 0})} V
	\]
	where the action is given by left multiplication. 
	\begin{defin}
		Let $V$ be a $\sld$-module. The $\mfL$-module $\mfT(V)$ is called \emph{generalized Verma module}.\\
		If $V$ is a finite-dimensional and irreducible $\sld$-module, we call $\mfT(V)$ \emph{minimal}.\\
		A minimal Verma module is called \emph{non-degenerate} if it is irreducible, \emph{degenerate} otherwise.
	\end{defin}
	
	\begin{rem}\label{remT(V)}
		Let us notice that, as vector spaces, $\mfT(V)\cong U(\mfL_-)\otimes V$. We will often use this isomorphism omitting the subscript $U(\mfL_{\geq 0})$  on the tensor product.\\
	\end{rem}
	When $V=V(\lambda)$ is an irreducible $\sld$-module of highest weight $\lambda$,  we may use the notation $\mfT(V) = \mfT(\lambda)$. A dominant weight $\lambda$ for $\sld$ will be expressed in terms of a quadruple $[a_1,a_2,a_3,a_4]\in \mathbb{N}^4$ where $\lambda = a_1 \omega_1 + \dotsb + a_4 \omega_4$ and $\omega_1 , \dotsc, \omega_4$ are the fundamental weights of $\sld$.\\
	One should pay special attention because, since we set $\mfd=(\C^5)^\ast$, all the highest weights modules are the duals of the $\mathfrak{sl}_5$ usual ones. For example, in our notation the $\mathfrak{sl}_5$ standard representation, which has highest weight $[1,0,0,0]$, will be $\mfd^\ast$.\\
	
	The grading of $U(\mfL_-)$ induces one on $\mfT(V)$.\\
	Poincar\'{e}-Birkhoff-Witt Theorem is still true in the superalgebra setting (see \cite[6.1]{M}), so fixed an ordered basis $\{\der_1\dotsc, \der_5, \xi_{12},\dotsc, \xi_{45}\}$ of $\mfL_-$, we can choose as a $PBW$-basis for $U(\mfL_-)$ the monomials $\der^{(I)}\xi^{K}$ where, 
	\begin{align}
	\der^{(I)} = \frac{\der_1^{i_1}}{i_1!}\cdots \frac{\der_5^{i_5}}{i_5!}\label{d^I}, \quad &I=(i_1, \dotsc, i_5)\in \mathbb{N}^5\\
	\xi^K =\xi_{12}^{k_{12}}\cdots\xi_{45}^{k_{45}}, \quad &K=(k_{12},\dotsc, k_{45})\in \{ 0,1\}^{10}.
	\end{align}
	The basis elements are, by definition of the grading, homogeneous of degree $p = 2 |I| + |K|$, where $|I|=i_1+ \dotsb + i_5$ and $|K|=k_{12}+\dotsb k_{45}$; they generate the homogeneous subspaces $U^p(\mfL_-)$. We can thus equip $\mfT(V)$ with a grading $\mfT^p(V)= U^p(\mfL_-)\otimes V$.\\
	We should notice that this grading and the grading of $\mfL$ are compatible, by which we mean that
	 \begin{equation}\label{compatiblegradings}
	 \mfL_n \mfT^p(V)\subseteq \mfT^{p-n}(V).
	 \end{equation}
	We will call the elements of $\mfT^p(V)$ \emph{homogeneous vectors of degree $p$}; in particular, we will call the degree $0$ ones \emph{constant}.\\
	For istance, $\mfT^0(V) = \C \otimes V$, $\mfT^1(V) = \mfs \otimes V$, $\mfT^2(V)= \mfd \otimes V + \bigwedge^2(\mfs) \otimes V$, $\mfT^3(V) = \mfd \mfs \otimes V + \bigwedge^3(\mfs)\otimes V$, etc.\\
	If $v\in S^n(\mfd)\bigwedge^m(\mfs)\otimes V$, if we want to keep track of the even and odd degrees, we will say that $v$ has degree $(n|m)$.\\
	
		Degeneracy of Verma modules can be reformulated in terms of \emph{singular vectors}.
	\begin{defin}
		Let $\mfT(V)$ be a Verma module. $v\in \mfT(V)$ is called a \emph{singular vector} if $\mfL_1v=0$.\\
		The space of singular vectors will be denoted by $sing \, \mfT(V)$.
	\end{defin}
	\begin{ex}
		Any constant vector $v\in \mfT^0(V)$ is singular, since in that case $\mfL_1 v \in \mfT^{-1}(V)=0$.
	\end{ex}
	\begin{rem}
		Take $v\in sing \, \mfT(V)$ and $z\in \mfL_0$. Then, for any $y\in \mfL_1$, $yv=0$ and so
		\[
		y(zv)=[y,z]v + z(yv) = [y,z]v=0
		\]
		where the last identity follows from the fact that $[\mfL_1,\mfL_0]\subseteq \mfL_1$ and the singularity of $v$.\\
		In other terms, $sing\, \mfT(V)$ is a $\mfL_0$-submodule of $\mfT(V)$. In particular, since $\mfT(V)=\bigoplus_p \mfT^p(V)$ as $\mfL_0$-modules, homogeneous components of a singular vector are singular. We will always assume that a singular vector is homogeneous.\\
		The same holds for the weight components of a singular vector. So we will, whenever possible, assume that a singular vector is also a weight vector.
	\end{rem}
	
	\begin{prop}\label{sing-degenerate}
		A minimal Verma module $\mfT(V)$ is degenerate if and only if it contains non constant singular vectors.
	\end{prop}
	\begin{proof}
		Assume $0\neq v \in sing \mfT^p(V)$ for some $p>0$. Since $\mfL_-$ and $\mfL_0$ can only respectively rise the degree of $v$ or preserve it and since $\mfL_+.v=0$ by assumption, $\mfL v$ is an $\mfL$-submodule of $\mfT(V)$, which is proper because it only contains vectors of degree $\geq p$.\\
		Vice versa, let $W\subset \mfT(V)$ be a non trivial proper $\mfL$-submodule and take $0\neq w \in W$. Since the action of $\mfL_1$ lowers strictly the degree of homogeneous components of $w$, we know that eventually $(\mfL_1)^{n}w=0$ for some finite $n\geq 1$; thus we can assume without loss of generalization that $w$ is singular. Now, if $w$ was constant, by irreducibility of $V$ we would have $\mfL_0 w = \C\otimes V$ and therefore, by iterated action of $\mfL_-$, we would obtain $\mfL w = \mfT(V)$. But $\mfT(V)\supsetneq W \supseteq \mfL w = \mfT(V)$, a contradiction. Hence $w$ is a non constant singular vector. 
	\end{proof}
	The proof of the proposition shows vividly how singular vectors "detect" degeneracy of minimal Verma modules.\\
	
	\begin{ex}\label{exsing}
		Let $V=V([0,0,0,1])\cong \mfd$ and let $ v = \nicesum{i} \xi_{1i} \otimes \der_i\in \mfT(\mfd) $.\\
		A generic element of $\mfL_1$ is of the form $y = x_h \xi_{kl}+x_k \xi_{hl}$ for some $h,k,l\in \{1,\dotsc, 5\}$.\\
		To check if $v$ is singular, we can carry out the computation:
		\begin{align*}
		y \cdot v = &\nicesum{i} x_h[\xi_{kl},\xi_{1i}]\otimes \der_i + x_k[\xi_{hl},\xi_{1i}]\otimes \der_i =\\
		&\nicesum{i} \varepsilon_{(kl1i)} \otimes (x_h\der_{(kl1i)})\der_i + \varepsilon_{(hl1i)} \otimes (x_k\der_{(hl1i)})\der_i=\\
		&\nicesum{i} -\varepsilon_{(kl1i)} \otimes e^h_{(kl1i)}\der_i - \varepsilon_{(hl1i)} \otimes e^k_{(hl1i)}\der_i=\\
		&\nicesum{i} -\varepsilon_{(kl1i)} \otimes \delta^h_i \der_{(kl1i)} - \varepsilon_{(hl1i)} \otimes \delta^k_i \der_{(hl1i)}=\\
		&-\varepsilon_{(kl1h)}\otimes \der_{(kl1h)} - \varepsilon_{(hl1k)}\otimes \der_{(hl1k)}=\\
		&-\varepsilon_{(kl1h)}\otimes \der_{(kl1h)} + \varepsilon_{(kl1h)}\otimes \der_{(kl1h)}=0 .
		\end{align*}
		So $v\in sing \, \mfT^1(V)$.	\footnote{Alternatively, one could have done the following:	notice that $v$ is an highest weight vector, therefore realize that it is sufficient to check that the lowest weight vector of $\mfL_1$, $x_5\xi_{45}$, acts trivially on $v$, which is clearly easier (or at least shorter) (see \cite[Ch.3]{CC}).}
	\end{ex}

	\section{Preliminaries on Lie Pseudoalgebras}\label{sectionPreliminariesLiepseudoalgebras}
	In this section we briefly give the definition of finite Lie pseudoalgebras and of their annihilation algebras and recall their main features following \cite[Section 2]{BDK2}. For a more detailed exposition, see also \cite{BDK1}.\\
	
	First of all we need a few preliminaries on Hopf algebras.\\
	Let $H$ be a cocommutative Hopf algebra with coproduct $\Delta$, counit $\epsilon$ and antipode $S$.\\
	We will use the Sweedler notation (see \cite{S}), for istance given $h\in H$:
	\begin{align*}
	\Delta(h) &=h_{(1)}\otimes h_{(2)};\\
	(\Delta\otimes id)\Delta(h) &= (id\otimes \Delta)\Delta(h)=h_{(1)}\otimes h_{(2)}\otimes h_{(3)};\\
	(id\otimes S)\Delta(h) &= h_{(1)}\otimes h_{(-2)}, \text{ etc.}
	\end{align*}
	Throughout the paper, $H=U(\mfd)$ will be the universal enveloping algebra of a Lie algebra $\mfd$ of dimension $n$.\\
	In this case the coproduct is given by $\Delta(\der)=\der\otimes 1 + 1\otimes \der$ and the antipode by $S(\der)=-\der$ for $\der\in \mfd$.\\
	Let $\der_1, \dotsc , \der_n$ be a basis of $\mfd$ and take the PBW basis of $H$, $\{\der^{(I)}\}_{I\in \mathbb{N}^{n}}$ as in \eqref{d^I}.\\
	With this choice of basis is easy to show that 
	\begin{equation}\label{deltaH}
	\Delta(\der^{(I)})= \nicesum{J+K=I} \der^{(J)}\otimes \der^{(K)}.\
	\end{equation}
	Moreover it lets us define an increasing filtration on $U(\mfd)$
	\begin{equation}
	F^{p} H =span_{\C}\{ \der^{(I)} \; |\; |I|\leq p \}.
	\end{equation}
	Now let $X= H^\ast := Hom(H,\C)$ and let $\{x_I\}$ be a dual basis of $\{ \der^{(I)}\}$, i.e. $\langle x_I, \der^{(J)}\rangle=\delta_I^J$. In particular we indicate the duals of the basis elements of $\mfd\subset H$, $\{\der_i\}$, with $\{x^i\}$, which provides a basis of $\mfd^\ast\subset X$.\\
	$X$ can be viewed as an $H$-bimodule with left and right actions given respectively by
	\begin{align}
	\langle hx, f \rangle &= \langle x, S(h)f\rangle;\label{leftactionHonX}\\
	\langle xh, f \rangle &= \langle x, f S(h)\rangle, \quad \text{for } x\in X,\,h,f\in H.\label{rightactionHonX}
	\end{align}
	Properties of $H$ reflect "dually" on $X$(\cite[Section 2]{BDK1}):
	\begin{itemize}
		\item cocommutativity of $H$ implies commutativity of $X$;\\
		\item from \eqref{deltaH} follows easily that $x_I x_J = x_{I+J}$;\\
		\item using the above, one can identify $X$ with the ring of former power series $\mathcal{O}_n=\C[[x^1, \dotsc, x^n]]$;\\
		\item setting $F_p X=(F^p(H))^{\perp}$ provides a decreasing filtration on $X$.
	\end{itemize}
	Considering the $\{F_p X\}$ as a fundamental system of neighborhoods one can define a topology on $X$ that makes it \emph{linearly compact} (\cite[Chapter 6]{BDK1}) (while considers the discrete topology on $H$). This makes the action of $H$ (and in particular of $\mfd$) on $X$ continuous.\\
	
	\begin{defin}
		A \emph{Lie ($H$-)pseudoalgebra} $L$ is a left $H$-module endowed with a map, called \emph{pseudobracket}
		\begin{equation*}
		[\cdot \ast \cdot] : L \otimes L \rightarrow (H\otimes H)\otimes_H L
		\end{equation*}
		which has the following properties:
		\begin{description}
			\item[H-bilinearity] $[fa\ast g b] = ((f\otimes g)\otimes_H 1)[a \ast b] \quad\forall a,b\in L,\; g,f\in H;$\\
			\item[Skew-commutativity] $[b\ast a] = - (\sigma\otimes_H 1)[a \ast b]\quad \forall a,b\in L$;\\
			\item[Jacobi] $[a \ast [b\ast c]]- ((\sigma\otimes 1)\otimes_H 1)[b \ast [a\ast c]]=[[a\ast b],\ast c]\quad \forall a,b,c\in L$.
		\end{description}
	where $\sigma:H\otimes H \rightarrow H\otimes$, $f\otimes g \mapsto g\otimes f$ is the permutation of factors and the composition of pseudobrackets in the Jacobi identity are suitable defined in $H^{\otimes 3}\otimes_H L$.\\
	A Lie pseudoalgebra is called \emph{finite} if it is finitely generated as a module over $H$.
	\end{defin}
	The name derives from the fact that this is an algebra in a \emph{pseudotensor category} as introduced in\cite{BD} (see also \cite[Chapter 3]{BDK1}).\\
	\begin{ex}
		Take a Lie algebra $\mathfrak{g}$. One can define the $\emph{current}$ Lie $H$-pseudoalgebra $Cur \, \mathfrak{g} = H\otimes \mathfrak{g}$ as free $H$-module with pseudobracket defined as
		\[
		[(1\otimes a)\ast (1\otimes b)]=(1\otimes 1)\otimes_H [a,b] \quad \text{ for } a,b\in \mathfrak{g}
		\]
		and extended then by $H$-bilinearity.
	\end{ex}
	We will focus on what are called \emph{primitive} Lie pseudoalgebras, which we will define in the next section.\\
	\begin{defin}
		A \emph{representation} of $L$, or \emph{$L$-module}, is a left $H$-module $M$ with an $H$-bilinear map
		\begin{equation*}
		\ast : L \otimes M \rightarrow (H\otimes H)\otimes_H M
		\end{equation*}
		such that
		\begin{equation*}
		[a\ast b]\ast m = a\ast (b\ast m) - ((\sigma\otimes 1)\otimes_H 1)(b\ast(a\ast m)) \quad \forall a,b\in L\, , m\in M.
		\end{equation*}
		An $L$-module is called \emph{finite} if it is finitely generated as an $H$-module.\\
		A subspace $N\subset M$ is an \emph{$L$-submodule} if $L\ast N\subset (H\otimes H)\otimes_H N$.\\
		An $L$-module is \emph{irreducible} if it does not contain any nontrivial proper submodules.
	\end{defin}

	The most important tool in the study of Lie pseudoalgebras (and their representations) is the annihilation algebra.\\
	
	Define $\mathcal{A}(L) := X\otimes_H L$, where as before $X=H^\ast$ and the right action of $H$ on $X$ is \eqref{rightactionHonX}.\\
	\begin{defin}
	Given $L$ a Lie $H$-pseudoalgebra, its \emph{annihilation algebra} is the Lie algebra $\mfL=\mathcal{A}(L)$ with Lie bracket
	\begin{equation}\label{annihilationbracket}
	[x \otimes_H a, y \otimes_H b]= \sum(x f_i)(y g_i)\otimes_H l_i \quad \text{ where } \quad [a \ast b]= \sum(f_i\otimes g_i)\otimes_H l_i.
	\end{equation}
	\end{defin}
	$H$ acts on $X\otimes_H L$ by \eqref{leftactionHonX} on the first factor. In this way $\mfd\subset U(\mfd)=H$ acts on $\mfL$ by derivations. The semidirect sum $\mfL^e=\mfd\rtimes \mfL$ is called \emph{extended annihilation algebra}.\\
	
	If $L$ is finite and $L_0$ is a finite-dimensional subspace that generates $L$ as a left $H$-module, we can define a filtration on $\mfL$ induced by the one of $X$, 
	\[F_p \mfL= \{ x\otimes_H a \in \mfL \, | \, x\in F_pX \text{ and } a\in L_0\},\]
	which satisfies:
	\begin{equation}\label{l-shift}
	[ F_n \mfL , F_p\mfL ] \subseteq F_{n+p-l} \mfL \; \text{  and  } \; \mfd(F_p \mfL) \subseteq F_{p-1}\mfL,
	\end{equation}
	where $l$ depends only on the choice of $L_0$. We can correct the $l$ shift by setting $\mfL_p = F_{p+l} \mfL$ so that $[\mfL_p,\mfL_n]\subseteq\mfL_{p+n}$. In particular, $\mfL_0$ is a Lie algebra.\\
	We can carry on the filtration to $\mfL^e$ setting $\mfL^e_p = \mfL_p$.\\
	
	An $\mfL^{e}$-module $V$ is called \emph{conformal} if any $v\in V$ belongs to some 
	\[ker_p V:=\{v\in V \, |\, \mfL_p v =0\}.\]
	
	The next result \cite[Proposition 2.1]{BDK2} we state will be crucial for our purposes.
	\begin{prop}\label{correspondance}
		Any module $V$ over a Lie pseudoalgebra $L$ has a natural structure of a conformal $\mfL^{e}$-module, given by the action of $\mfd$ on $V$ and by
		\begin{equation}\label{(correspondance)}
		(x\otimes_H a) v = \sum\langle x, S(f_i g_{i_{(-1)}})\rangle g_{i_{(2)}}v_i \quad \text{where } a\ast v = \sum (f_i \otimes g_i)\otimes_H v_i
		\end{equation}
		for $a\in L$, $x\in X$, $v\in V$.\\
		Conversely, any conformal $\mfL^{e}$-module $V$ has a natural structure of an $L$-module given by
		\begin{equation}
		a \ast v = \nicesum{I\in \mathbb{N}^n} (S(\der^{(I)})\otimes 1)\otimes_H((x_I\otimes_H a)\cdot v)
		\end{equation}
		Moreover, $V$ is irreducible as a module over $L$ if and only if it is irreducible as a module over $\mfL^e$ (or $\mfL$).
	\end{prop}
\newpage
	
	\section{Primitive Lie pseudoalgebras of type W and S}\label{sectionPrimitiveLiepseudoalgebras}
	We now define the Lie pseudoalgebra $W(\mfd)$ and its subalgebra $S(\mfd)$ and apply the constructions of the previous section.
	Again, we will follow \cite[]{BDK2}.\\
	\begin{defin}
		The Lie pseudoalgebra $W(\mfd)$ is the free $H$-module $H\otimes \mfd$ with pseudobracket
		\begin{equation}
		[(f\otimes a)\ast (g\otimes b)] = (f \otimes g)\otimes_H(1 \otimes [a,b])-(f \otimes ga)\otimes_H(1\otimes b) + (fb\otimes g)\otimes_H(1\otimes a)
		\end{equation}
	\end{defin}
	Define the $H$-linear map $div: W(\mfd)\rightarrow H$ by $div(\sum h_i \otimes \der_i) = \sum h_i\der_i$. Then
	\[
	S(\mfd) := \{ s\in W(\mfd) \, | \, div(s)=0\}
	\]
	is a subalgebra  of the Lie pseudoalgebra $W(\mfd)$.\\
	In \cite[Proposition 8.1]{BDK1} it is shown that $S(\mfd)$ is generated as an $H$-module by the elements of the form:
	\begin{equation}\label{s_ab}
	s_{ab} = a\otimes b - b\otimes a - 1\otimes [a,b] \quad \text{ for } a,b\in \mfd.
	\end{equation}
	Let $\mathcal{W}=\mathcal{A}(W(\mfd))$ be the annihilation algebra of $W(\mfd)$. Since $W(\mfd)=H\otimes \mfd$, 
	\[\mathcal{W}= X\otimes_H (H\otimes \mfd)\equiv X\otimes \mfd.\]
	The Lie bracket of $\mathcal{W}$ is obtained from the pseudobracket of $W(\mfd)$ by \eqref{annihilationbracket}:
	\[
	[x\otimes a, y \otimes b]=xy \otimes[a,b] - x(ya)\otimes b + (xb)y\otimes a \quad \text{for } a,b\in \mfd, \, x,y\in X.
	\]
	The action of $H$ on $\mathcal{W}$ is given by $h(x\otimes a)=hx\otimes a$ and $\mfd$ acts on $\mathcal{W}$ by derivations.\\
	
	Since $W(\mfd)$ is a free $H$-module, we can choose $L_0=\C\otimes \mfd$ and obtain the induced decreasing filtration on $\mathcal{W}$
	\[
	\mathcal{W}_p=F_p\mathcal{W}=F_p X \otimes_H L_0 \equiv F_p X \otimes \mfd.
	\]
	$\mathcal{W}_{-1}=\mathcal{W}$ and it satisfies \eqref{l-shift} for $l=0$; notice also that $\mathcal{W}/\mathcal{W}_0\cong \C\otimes \mfd$ and $\mathcal{W}_0/\mathcal{W}_1\cong\mfd^*\otimes \mfd$.\\
	$H$ can be endowed with a $W(\mfd)$-pseudoaction given by $(f\otimes a) \ast g= -(f \otimes ga)\otimes_H 1$ where $f,g\in H, \, a\in \mfd$; this induces an action of $\mathcal{W}=X\otimes_H W(\mfd)$ on $X\otimes_H H\equiv X$: 
	\begin{equation}\label{actionW(d)onX}
	(x\otimes a)y = -x(ya) \quad \text{ for } a\in \mfd, \, x,y\in X.
	\end{equation}
	Using the fact that $X$ can be identified as $\mathcal{O}_n$ (compatibly with corresponding filtrations and topologies) and that $\mfd$ acts on $X$ by continuous derivations, we can make $\mathcal{W}$ act on $\mathcal{O}_n=\C[[t_1, \ldots t_n]]$ by continuous derivations too. This way we are defining a Lie algebra homomorphism $\varphi:\mathcal{W}\rightarrow W(n)$ where 
	\[
	W(n)=Der(\mathcal{O}_n)=\Big\{\nicesum{i=1}^n f_i \frac{\der}{\der t_i} \,|\, f_i\in \C[[t_1, \dotsc, t_n]] \Big\}
	\]
	$W(n)$ has a natural filtration given by 
	\[
	F_p W(n)=\Big\{\ \nicesum{i} f_i \frac{\der}{\der t_i} \, | \, f_i\in \C[[t_1, \dotsc, t_n]]_k \, , \, k\leq p \Big\}
	\]
	where $\C[[t_1, \dotsc, t_n]]_k$ is the homogeneous component of degree $k$.\\
	
	It is proven in \cite{BDK2} that holds the following:
		\begin{enumerate}
			\item $\varphi$ is an isomorphism of Lie algebras;\\
			\item $\varphi(x \otimes a)=\varphi(x)\varphi(a)$ $\forall x\in X, \, a\in \mfd$;\\
			\item $\varphi(1\otimes \der_i)= - \frac{\der}{\der t_i} $ mod $F_0 W(n)$;\\
			\item $\varphi(\mathcal{W}_p)= F_p W(n)$ $\forall p \geq -1$.\\
		\end{enumerate}

	In what follows we will assume that $dim \mfd = n >2$ (which is okay for us since we will apply it for $n=5$).\\
	Let $\mathcal{S }= \mathcal{A}(S(\mfd)) = X\otimes_H S(\mfd)$ be the annihilation algebra of the Lie pseudoalgebra $S(\mfd)$.\\
	The Lie bracket is the one of $\mathcal{W}$, since the canonical injection of $S(\mfd)$ into $W(\mfd)$ induces a Lie algebra homomorphism $\iota: \mathcal{S}\hookrightarrow \mathcal{W}$.\\
	Explicity, if $s=\sum h_i \otimes \der_i \in S(\mfd)\subset W(\mfd) = H\otimes \mfd$,
	\[
	\iota(x\otimes_H s)=\sum xh_i \otimes \der_i \in\mathcal{W}\equiv X\otimes \mfd
	\]
	Choosing $L_0=span_{\C}\{s_{ab} | \, a,b\in \mfd\}$ (where $s_{ab}$ are the ones defined in \eqref{s_ab}) we get a decreasing filtration of $\mathcal{S}$:
	\[
	\mathcal{S}_p=F_{p+1} \mathcal{S}=F_{p+1} X \otimes_H L_0 \quad \text{ for } p\geq-2.
	\]
	$\mathcal{S}_{-2}=\mathcal{S}$ and it satisfies \eqref{l-shift} for $l=1$.\\
	In \cite[Section 8.4]{BDK1} is proven that $\mathcal{S}\cong S(n)$. But we would also like all the filtrations and related topologies defined on these spaces to be compatible.
	In order to do so, we would like to use $\varphi$ defined before for $\mathcal{W}$ which behaves well related to the filtrations.\\
	This can be done but carefully.\\
	First define a map $div:\mathcal{W}\rightarrow X$ as $div(\sum y_i \otimes \der_i)=\sum y_i \der_i$.\\
	It is not difficult to verify that $div([A,B])= A \, div(B) - B\, div(A)$ $\forall A,B\in \mathcal{W}$ (where the action of $\mathcal{W}$ on $X$ is \eqref{actionW(d)onX}), which implies that
	\[
	\overline{\mathcal{S}}=\{ A\in \mathcal{W} \, | \, div(A)=0\}
	\]
	is a Lie subalgebra of $\mathcal{M}$.\\
	In \cite[Section 3.4]{BDK2} is proven first that $\iota:\mathcal{S}\xrightarrow{\sim}\overline{\mathcal{S}}$ in such a way that $\iota(\mathcal{S}_p)=\overline{\mathcal{S}}\cap\mathcal{W}_p$ $\forall p\geq-1$ [Proposition 3.5], then that $\phi$ maps $\overline{\mathcal{S}}$, up to a Lie algebra automorphism $\psi$ induced by a ring automorphism of $\mathcal{O}_n$, to $S(n)\subset W(n)$ [Proposition 3.6].\\
	Finally a Lie algebra isomorphism 
	\begin{equation}\label{isoS(n)}
	\phi:\mathcal{S}\xrightarrow{\sim} S(n)\subset W(n)
	\end{equation}
	 such that $\mathcal{S}_p$ maps onto $S(n)\cap F_p W(n)$ is obtained [Corollary 3.3]. In particular we have that $\mathcal{S}_{-2}=\mathcal{S}_{-1}=\mathcal{S}$.\\
	
	\begin{rem}\label{remark verma modules as S(d) modules}
		Take a generalized Verma module $\mfT(V)=U(\mfL_{-})\otimes V = U(\mfd + \mfs)\otimes V$. It is in particular an $S(5)$-module and, in view of \eqref{isoS(n)}, also an $\mathcal{S}$-module. Furthermore, considering the action of $\mfd$ as left multiplication in $U(\mfd + \mfs)$, we can view $\mfT(V)$ as a $\mathcal{S}^e$-module. Since all these identifications are compatible with the filtrations, \eqref{compatiblegradings} implies that $\mfT(V)$ is a conformal $\mathcal{S}^e$-module. By Proposition \ref{(correspondance)}, it has a natural structure of $S(\mfd)$-module.
	\end{rem}

Finally, we define what are called \emph{tensor modules} for $W(\mfd)$ and $S(\mfd)$ (which are analogs of Verma modules) and extrapolate from \cite{BDK2} what we need.\\

Recall that $\mathcal{W}_0/\mathcal{W}_1\cong \gld$. So if we take a $\gld$-module $V$, we can allow $\mathcal{W}_1$ to act on it trivially and get a $\mathcal{W}_0$-module. After that, we can induce so to get a $\mathcal{W}$-module.\\
But in order to correlate this with the Lie pseudoalgebra $W(\mfd)$, we need to take account of the action of $\mfd$. To do so, we consider the extended annihilation algebra $\mathcal{W}^e$.\\
We call $\mathcal{N}_{\mathcal{W}}$ the normalizer of $\mathcal{W}_p$ in $\mathcal{W}^e$. In \cite{BDK2} is proven that it is independent of $p$ and that $\mathcal{W}^e=\mfd\oplus\mathcal{N}_{\mathcal{W}}$[Proposition 3.3]. Moreover, $\mathcal{W}_1$ acts trivially on any irreducible finite-dimensional conformal $\mathcal{N}$-module (i.e. modules for which every element is killed by some $\mathcal{W}_p$) and $\mathcal{N}_{\mathcal{W}}/\mathcal{W}_1\cong \mfd \oplus \gld$, so that we have a one-to-one correspondence between irreducible finite-dimensional $\mfd\oplus \gld$-modules and irreducible finite-dimensional conformal $\mathcal{N}_{\mathcal{W}}$-modules [Proposition 3.4].\\

In \cite[Section 3.5]{BDK2} totally analogous results are proven for $\mathcal{S}$, whereas $\mathcal{N}_\mathcal{S}/\mathcal{S}_1\cong \mfd \oplus \sld$.\\

Take a finite-dimensional $\mfd \oplus \gld$-module $V$; letting $\mathcal{W}_1$ act as zero on it, we can define an action of $\mathcal{N}_\mathcal{W}$ and then define the $\mathcal{W}^e$-module $T(V)=Ind_{\mathcal{N}_{\mathcal{W}}}^{\mathcal{W}^e} V =\mathcal{W}^e\otimes_{\mathcal{N}_{\mathcal{W}}} V$ which can be identified as an $H$-module with $H \otimes V$ since $\mathcal{W}^e=\mfd \oplus \mathcal{N}_{\mathcal{W}}$.\\
$T(V)$ is called a \emph{tensor module} for $W(\mfd)$.\\
\begin{defin}
	Let $\mathfrak{g}_1$ and $\mathfrak{g}_2$ be Lie algebras and let $V_i$ be $\mathfrak{g}_i$-modules fro $i=1,2$.\\
	We indicate with $V_1\boxtimes V_2$ the $\mathfrak{g}_1\oplus \mathfrak{g}_2$-module $V_1\otimes V_2$ where $\mathfrak{g}_i$ only acts on the $V_i$ factor.
\end{defin}
If $V$ is of the form $\Pi\boxtimes U$, we will also indicate $T(V)=T(\Pi, U)$.\\
We can define on a tensor module $T(V)=H\otimes V$ a filtration as follows:
\begin{equation}
F^p T(V) = F^p H \otimes V \quad\text{for } p\geq -1
\end{equation}
which behaves nicely relatively to the filtration of $\mathcal{W}$:
\begin{lem}\cite[Lemma 6.3]{BDK2}
	For every $p\geq0$ we have:
	\begin{enumerate}
		\item $\mfd \cdot F^p T(V) \subset F^{p+1} T(V)$;
		\item $\mathcal{N}_{\mathcal{W}}\cdot F^p T(V)\subset F^p T(V)$;
		\item $\mathcal{W}_1 \cdot T(V)\subset F^{p-1}T(V)$.
	\end{enumerate}
\end{lem}

For $S(\mfd)$ we do the same construction:\\
take a finite-dimensional $\mfd \oplus \sld$-module $V$, let $\mathcal{S}_1$ act trivially on it so that it has an action of $\mathcal{N}_\mathcal{S}$; then consider the $\mathcal{S}^e$-module $T_{S}(V)=Ind_{\mathcal{N}_{\mathcal{S}}}^{\mathcal{S}^e} V =\mathcal{S}^e\otimes_{\mathcal{N}_{\mathcal{S}}} V$ which again can be identified, as an $H$-module, with $H\otimes V$.\\
In \cite[Theorem 7.3]{BDK2} it is proven that these modules can be obtained as the restriction of tensor modules for $W(\mfd)$, therefore we will call them again \emph{tensor modules} for $S(\mfd)$.\\
If $V$ is of the form $\Pi\boxtimes U$, we will also indicate $T_{S}(V)=T_S(\Pi,U)$.\\
We can define the same filtration we defined in the $W(\mfd)$ case and have the same nice behavior relatively to the filtration of $\mathcal{S}$.\\

It makes sense now to introduce the notion of \emph{singular vectors} for $W(\mfd)$ and $S(\mfd)$.
 
\begin{defin}
	For a $W(\mfd)$-module $V$, a \emph{singular vector} is an element $v\in V$ such that $\mathcal{W}_1 \cdot v =0$. The space of singular vectors in $V$ in indicated by $sing \, V$.\\
	Analogously one defines singular vectors for $S(\mfd)$-modules.
\end{defin}

\begin{rem}
	By Remark \ref{remark verma modules as S(d) modules} a Verma module has a structure of $S(\mfd)$-module. Again we stretch the fact that all the identifications preserve the filtrations. Therefore, by the last definition, a singular vector for $E(5,10)$ in a Verma module is also  singular for $S(\mfd)$.
\end{rem}

Focusing only on type $S$ now, we summerize all the main results about tensor modules for $S(\mfd)$ and singular vectors in \cite[Section 7]{BDK2}.\\
We will indicate by $\Omega^n$ the $\sld$-module $\bigwedge \mfd^\ast$ with the natural action, where $\Omega^0$ is the trivial module $\C$.
\begin{theo}\label{theorem degrees tensor modules}$ $\newline
	\begin{itemize}
		\item Every irreducible finite $S(\mfd)$-module is a quotient of a tensor module;\\
		\item  Let $\Pi$ (resp. $U$) be an irreducible finite-dimensional module over $\mfd$ (resp. $\sld$). Then the $S(\mfd)$-module $T_S(\Pi,U)$ is irreducible if and only if $U$ is not isomorphic to $\Omega^n$ for any $n\geq 0$;\\
		\item If $V=T_S(\Pi,\Omega^n)$, $n\neq 1$, then $sing \, V \subset F^1 V$;\\
		\item If $V=T_S(\Pi,\Omega^1)$, then $F^1 V \subsetneq sing \, V \subset F^2 V$.\\
	\end{itemize}
\end{theo}

Combining the results above, one gets the following picture:\\
if one looks for irreducible $S(\mfd)$-modules (or, by mean of Proposition \ref{correspondance}, irreducible $\mathcal{S}^e$-modules ) one studies tensor modules. For a tensor module, to be irreducible is equivalent to not contain positive degree singular vectors. Singular vectors can only be found when one induces from irreducible $\sld$-modules $\Omega^n$ and they are found in degree $1$. except for the case $n=1$ where one finds also singular vectors of degree 2.\\

\section{Bound on degree of singular vectors}\label{sectionBoundSingularvectors}
Direct computation of singular vectors is not immediate, especially in high degrees. One thing one can do is trying, as a start, to rule out as many options as possible. This can be done, for example, by looking for restricting conditions on the degree of singular vectors, which is what we are about to do.\\
Recall that by Theorem \ref{theorem degrees tensor modules}, we have a very important result that points in this direction in the "pseudo" setting.\\
Turns out, in order to get a first bound on the degree of singular vectors, it is enough to take into account just the even structure of $\mfL$, exploiting the pseudoalgebra techniques available.\\
To do so, we first define some subspaces of $\mfT(V)$.\\

For $i=0,\dotsc, 10$, let
\begin{equation}
\Gamma_i(V):=\Big\{ v=\sum_{I,K}\der^{(I)}\xi^K \otimes v_{IK}\in \mfT(V) \, | \, |K|\leq i \text{ se } v_{IK}\neq 0 \Big\}.
\end{equation}
In other words, $\Gamma_i(V)$ consists of vectors with "odd degree" at most $i$. We may also describe $\Gamma_i(V)$ as the space generated by the PBW monomials of degree $(n\, |\,j)$ with $j\leq i$.\\
It is easy to check the following properties:
\begin{itemize}
 	\item $\Gamma_i(V)\subseteq \Gamma_{i+1}(V)$;\\
 	\item $\Gamma_0(V)=U(\mfd)\otimes V$;\\
 	\item  $\Gamma_{10}(V)=\mfT(V)$;\\
 	\item  $\Gamma_i(V)$ is a $\mfL_{(0)}\cong S(5)$-submodule of $\mfT(V)$.
 \end{itemize}
Basically, we have built a finite  filtration of $S(5)$-modules of $\mfT(V)$.\\
The last property follows from the fact that the action of an even element in $\mfL$ can only lower the odd degree in $\mfT(V)$. Take for example $y\in \mfL_2, \, v=\der^{(I)}\xi^{K}\otimes v_I^K\in \mfT(V)$.\\
We have
\[
y\cdot v=y\cdot (\der^{(I)}\xi^K\otimes v_{I}^{K})=[y , \der^{(I)}] \xi^K \otimes v_{I}^{K}+ \der^{(I)} [y, \xi^K] \otimes v_{I}^{K}+\der^{(I)}\xi^K \otimes  y\cdot v_{I}^{K} .
\]
Here the third term is $0$ because $\mfL_+$ acts trivially on $V$; the first term consists of elements of degree $(\, |I|-1 \; | \; |K|\, )$; lastly, the second term, once expanded the bracket and sorted everything, can only contribute with elements of degree $(|I| \; | \; |K|-2 )$ or $(|I|+1 \; | \; |K| -4)$. In any case the odd degree cannot increase.\\

These properties allow us to talk about quotients.\\
Let us consider the quotients of $S(5)$-modules $\Gamma_{i}(V)/\Gamma_{i-1}(V)$ for $i=0,\dotsc, 10$ (where we impose $\Gamma_{-1}(V)=0$).\\
As $\sld$-modules, they are isomorphic to $U(\mfd)\otimes (\bigwedge^{i}(\mfs)\otimes V)$ (follows from \cite[Corollary 6.4.5]{M}). The latter look a lot like tensor modules $T(\bigwedge^{i}(\mfs)\otimes V)$ for $S(\mfd)$: if the action of $S(5)$ can be interpreted as the action of the annihilator algebra associated to pseudoaction of $S(\mfd)$, we can put to use Proposition \ref{(correspondance)}.
This is, in fact, possible in the following way.\\

The action of $y\in \mfL_{2j}$ on a class $\overline{u \cdot \xi^K \otimes v}\in \Gamma_{i}(V)/\Gamma_{i-1}(V)$, where $u\in U(\mfd)$, $|K|=i$ and $v\in V$, behaves, depending on $j$, like:\\
\begin{description}
	\item[for $\mathbf{j=-1}$]$y\cdot \overline{u\xi^K \otimes v} =\overline{{(yu)\xi^K \otimes v}}$, since in this case $y\in \mfd\subseteq U(\mfd)$;\\
	\item[ for  \, $\mathbf{j=0}$]$y\cdot \overline{u\xi^K \otimes v} = \overline{[y,u]\xi^K\otimes v + u [y,\xi^K]\otimes v + u \xi^K \otimes (y\cdot v)}$;\\
	\item[ for  \, $\mathbf{j>0}$]$\begin{aligned}[t] y\cdot \overline{u\xi^K \otimes v} =& \overline{[y,u]\xi^K\otimes v + u [y,\xi^K]\otimes v + u \xi^K \otimes (y\cdot v)}\\
																		=&\overline{[y,u]\xi^K\otimes v}
						  \end{aligned}$\\
\end{description}
where the last equality is due to the fact that $y$ lowers the odd degree of at least $1$, sending the second term to $0$ in the quotient, and acts trivially on $V$.\\
Notice that in the case $j=0$, the action of $y\in\mfL_0\cong \sld$ on $\xi^K$ is actually the same as the one on the $\sld$-module $\bigwedge^{i}(\mfs)$ (since the other terms that usually appear in the bracket $[y,\xi^K]$ once sorted are $0$ in the quotient).\\

Summing up, we have that $\mfL_{(0)}=S(5)$ acts on $\Gamma_{i}(V)/\Gamma_{i-1}(V)$ 
\begin{itemize}
	\item by left multiplication on the $H=U(\mfd)$ factor with the negative degree part;
	\item by the natural action of $\sld$ on $U(\mfd)\otimes(\bigwedge^i(\mfs)\otimes V)$ with the degree $0$ part;
	\item trivially on $\bigwedge^{i}(\mfs)\otimes V$ with the positive degree part.
\end{itemize}
Since this is exactly the action of $\mathcal{S}$ on the tensor module $T_S(\bigwedge^i(\mfs)\otimes V)$ via the isomorphism in \eqref{isoS(n)}, we can state:
\begin{theo}\label{theorem Gamma tensor}
	Let $V$ be a finite-dimensional irreducible $\sld$-module. Then we have an isomorphism of $S(5)\cong X\otimes_H S(\mfd)$-modules
\begin{align}
\varphi:\Gamma_{i}(V)/\Gamma_{i-1}(V) &\xrightarrow{\sim} T_S(\textstyle\bigwedge^{i} (\mfs)\otimes V)
\end{align}
\end{theo}

Take now $v\in sing \, \mfT^p(V)$. It is, in particular, a singular vector for $\mfL_{(0)}=S(5)$ of degree $p$. We indicate the space of such vectors with $sing_{S(5)}\, \mfT(V)$.\\
If we consider $\overline{v}\in\Gamma_{i}(V)/\Gamma_{i-1}(V)$ for a suitable $i=0, \dotsc, 10$, it will still be a singular vector in what we now know is a tensor module. Therefore, by Theorem \ref{theorem degrees tensor modules}, the even degree of $\overline{v}$ must be $\leq 2$. Since the odd degree of a vector cannot be larger than $10$, these ideas, formalized, prevent vectors with a sufficiently high enough degree from being singular.
\begin{theo}\label{theo14}
	Let $\mfT(V)$ be a minimal Verma module and let $v\in sing \, \mfT(V)$.  Then $v$ has degree at most $14$.
\end{theo}
\begin{proof}
We can assume that $v$ is homogeneous of degree $p$.\\
We have, if $p$ is either even or odd:\\

\begin{description}
	\item[p=2n]   $\mfT^p(V)=S^n(\mfd)\otimes V +S^{n -1}(\mfd) \, \bigwedge^2(\mfs)\otimes V+\dotsb + S^{n-5}(\mfd) \bigwedge^{10}(\mfs)\otimes V$;
	\item[p=2n+1]   $\mfT^p(V)=S^{n}(\mfd) \, \mfs \otimes V + S^{n-1}(\mfd)\bigwedge^3(\mfs)\otimes V + \dotsb + S^{n-4}(\mfd) \bigwedge^9(\mfs)\otimes V$.\\
\end{description}

We study the case $p=2n$.\\
Let $0\leq m_0 \leq 10$ be the greatest index such that the term of $v$ in degree $(\frac{p-m_0}{2} | m_0)$ is not $0$ (i.e. the term of $v$ in $S^{(p-m_0)/2}(\mfd)\bigwedge^{m_0}(\mfs)\otimes V$).\\
Therefore $v\in\Gamma_{m_{0}}(V)$ and it is a combination of terms in degrees \[( \textstyle\frac{p-m_0}{2}  |  m_0), \, ( \textstyle\frac{p-m_0}{2} +1  | m_0 -2  ), \dotsc , \, ( \textstyle\frac{p}{2}  |  0  ).\]
We can then consider 
\[
\overline{v} \in \Gamma_{m_{0}}(V)/\Gamma_{m_{0}-1}(V)\cong U(\mfd)\otimes(\textstyle\bigwedge^{m_0}(\mfs)\otimes V).
\]
Notice that given $y\in\mfL_{2}$, $y\cdot v\in \mfT^{p-1}(V)$ and the term of degree $(\frac{p-m_0-1}{2} | m_0)$ can be obtained only acting with $y$ on the term of $v$ of degree $(\frac{p-m_0}{2} | m_0)$. Hence, if $v$ is singular for $S(5)$, so must be $\overline{v}$.\\
Recapitulating: if $v\in sing \,\mfT^p(V)$, $\overline{ v }\in U(\mfd)\otimes(\bigwedge^{m_0}(\mfs)\otimes V)=T(\bigwedge^{m_0}(\mfs)\otimes V)$ is a singular vector for $S(5)$ of (even) degree $(p-m_0)/2$ in a tensor module. By Theorem \ref{theorem degrees tensor modules}, the even degree of $\overline{ v }$ must be less than or equal to $2$, which means that $(p-m_0)/2\leq 2$, that is $p\leq m_0 + 4\leq 14$. When $p$ is odd the same argument holds.
\end{proof}

The proof actually tells us more than the statement of the theorem: we can not only estimate the singular vectors' degree, but we can also rule out straightforwardly most of the irreducible $\sld$-modules whose induced modules we expect to possibly contain singular vectors of a certain degree.\\
We recall that in our notation $\Omega^1=\mfd^\ast\cong V([1,0,0,0])$. Similarly $\Omega^2\cong V([0,1,0,0])$, $\Omega^3\cong V([0,0,1,0])$ and $\Omega^4\cong V([0,0,0,1])$.\\

Now apply, for example, the proof's arguments on degree 14:\\
let $v$ be a singular vector of degree $14$ in a minimal Verma module $\mfT(V)$ and consider $\overline{v}\in \Gamma_{10}(V)/\Gamma_{9}(V)\cong U(\mfd)\otimes(\bigwedge^{10}(\mfs)\otimes V)\cong U(\mfd)\otimes V$. If $\overline{ v }\neq 0$, this means that the term of $v$ of degree $(2 | 10)$ is not $0$.\\
We know by Theorem \ref{theorem degrees tensor modules} that we can find singular vectors of (even) degree $2$ in a tensor module $T(V)$ where $V$ is irreducible if and only if $V\cong \Omega^1\cong V([1,0,0,0])$.\\
Therefore, if we assume that $V\ncong \mfd^{\ast}$, $\overline{ v }$ will necessarily be $0$. This implies that $v\in \Gamma_{8}(V)$ and we can consider $\overline{ v }\in U(\mfd)\otimes(\bigwedge^8(\mfs)\otimes V)$, thus obtaining a singular vector in $T(\bigwedge^8(\mfs)\otimes V)$ of (even) degree $(14-8)/2=3$. This cannot happen, so that the only possible solution is $\overline{ v }=0$. Iterating, we discover that $v$ must be $0$.\\
We proved:
\begin{lem}
	If $V\ncong \mfd^\ast\cong V([1,0,0,0])$, $sing \, \mfT^{14}(V)=\lbrace 0 \rbrace$.
\end{lem}
\section{Bound refining}\label{sectionBoundRefining}
	 An extremely simple lemma will be extremely useful:
	 \begin{lem}\label{easylemma}
	 	Let $\mfT(V)$ be a Verma module for $\mfL$. If $v\in sing \, \mfT(V)$ and $\xi\in \mfL_{-1}\cong \mfs$, then $\xi v \in sing_{S(5)} \, \mfT(V)$.
	 \end{lem}
	\begin{proof}
		Take $y\in \mfL_{2}$. Then
		\[
		y \cdot (\xi v) = [y,\xi]\cdot v + \xi (y\cdot v)=0,
		\]
		where the second term is $0$ because $v$ is singular. The same goes for the first term since $[y,\xi]\in \mfL_{1}$.
	\end{proof}
	We apply this new piece of information to the case $V\cong \mfd^*$.\\
	Let $v\in sing \, \mfT^{14}(\mfd^{\ast})$. By the lemma, given any $\xi\in \mfs$, $\xi v$ (which has now degree $15$), is still singular for the action of $S(5)$. Consider the term of degree $(3 | 9)$ and the corresponding $\overline{ \xi v }\in \Gamma_{9}(V)/\Gamma_{8}(V)\cong T(\bigwedge^9(\mfs)\otimes V)$. Like before, it is still singular and has even degree $3$, which implies that $\overline{ \xi v }=0$. We can then consider $\overline{ \xi v }\in \Gamma_{7}(V)/\Gamma_{6}(V)$ and, iterating the argument, obtain that $\xi v$ must be $0$ $\forall \xi \in \mfs$.\\
	We remark that $\mfL_{-2}=[\mfL_{-1},\mfL_{-1}]$ which means that, for any $\der\in \mfd \cong \mfL_{-2}$, we can find $\xi_1,\xi_2\in\mfs\cong\mfL_{-1}$ such that $\der=[\xi_1,\xi_2]$. This in particular implies that $\der v =0$ $\forall \der\in \mfL_{-2}$. Since the action of $\mfL_{-2}$ on a tensor module is simply given by left multiplication, it can only mean that $v=0$.\\
	In conclusion, we showed that even if $V\cong \mfd^\ast$, $\mfT(V)$ cannot have singular vectors of degree 14.\\
	These ideas can be applied systematically to perform a refining of the bound in Theorem \ref{theo14}.
	
	\begin{theo}\label{raff}
		Let $\mfT(V)$ be a minimal Verma module and let $v\in sing \, \mfT(V)$. Then $v$ has degree $\leq 12$. More precisely:
		\begin{enumerate}
			\item if $V\ncong V([0,0,1,0])$, singular vectors have degree at most $11$;\\
			\item if $V\ncong V(\lambda)$ where $\lambda=[0,0,0,1],\, [0,0,1,0],\,[0,1,0,0], \, [1,0,0,0], \, [0,1,1,0],$ or $[1,0,0,1]$, singular vectors have degree at most $10$;\\
			\item $\begin{aligned}[t]
				if V\ncong V(\mu) \text{ where } \mu = &[0,0,0,0],\,[1,0,0,0],\,[0,0,0,1],\,[0,0,1,0],\, [0,1,0,0], \\
				 									   &[1,1,0,0], \, [0,1,1,0],\, [1,0,0,1], \,[0,0,1,1],\,[1,0,1,0], \\
													   &[0,1,0,1], \, [1,1,0,1],\,[0,2,0,0],\,[2,0,0,0], \, [1,0,2,0], \\
													    &\text{or } [3,0,0,1], \text{ singular vectors have degree at most $9$.}
			\end{aligned}$
		\end{enumerate}
	\end{theo}
	The proof revolves around arguments similar to the previous ones. To that end, because of Theorem \ref{theorem Gamma tensor} and Theorem \ref{theorem degrees tensor modules}, we will need to be able to determine when, given an irreducible $\sld$-module $V$, we can (or rather cannot) find a copy of $V(\omega_i)$ in $\bigwedge^j(\mfs)\otimes V$.\\
	Recall that every irreducible $\sld$-module $V$ has a highest weight vector and that $V$ is uniquely determined by the highest weight. Here, as before, $\omega_1=[1,0,0,0]$, $\omega_2=[0,1,0,0]$, $\omega_3=[0,0,1,0]$, $\omega_4 = [0,0,0,1]$ and $\omega_0 = [0,0,0,0]$.\\
	By Frobenius duality (keeping in mind that these are all finite-dimensional modules), $V(\omega_i)\subseteq \bigwedge^j(\mfs)\otimes V$ if and only if $V(\omega_0)\subseteq \bigwedge^j(\mfs)\otimes V \otimes V(\omega_i)^*$ se e solo se $V^*\subseteq \bigwedge^j(\mfs)\otimes V(\omega_i)^*$ se e solo se $V\subseteq \bigwedge^j(\mfs^*)\otimes V(\omega_i)$.\\
	In the following table we list the highest weights of the irreducible representations that appear in the decomposition of those tensor products. It was obtained using computer software "LiE" ( see \cite{LiE} for further informations).\\	
	
	\begin{tabular}{|c|p{3cm}|p{3cm}|p{3cm}|p{3cm}|} 
		\hline
		& $\bigwedge^9(\mfs^*)\otimes V(\omega_i)$ & $\bigwedge^8(\mfs^*)\otimes V(\omega_i)$ & $\bigwedge^7(\mfs^*)\otimes V(\omega_i) $
		& $\bigwedge^6(\mfs^*)\otimes V(\omega_i)$\\
		\hline
		i=0 & [0,1,0,0] & [1,0,1,0] & [0,0,2,0], [2,0,0,1] & [1,0,1,1], [3,0,0,0]\\
		\hline
		i=1 & [0,0,1,0], [1,1,0,0] & [0,1,1,0], [1,0,0,1], [2,0,1,0] & [0,0,1,1], [1,0,2,0], [1,1,0,1], [2,0,0,0], [3,0,0,1] & [0,1,1,1], [1,0,0,2], [1,0,1,0], [2,0,1,1], [2,1,0,0], [4,0,0,0] \\
		\hline
		i=2 &[0,0,0,1], [0,2,0,0], [1,0,1,0] & [0,0,2,0], [0,1,0,1], [1,0,0,0], [1,1,1,0], [2,0,0,1] & [0,0,1,0], [0,1,2,0], [1,0,1,1], [1,1,0,0], [2,1,0,1], [3,0,0,0] & [0,0,2,1], [0,1,0,2], [0,1,1,0], [1,0,0,1], [1,1,1,1], [2,0,0,2], [2,0,1,0], [3,1,0,0]\\
		\hline
		i=3 &[0,0,0,0], [0,1,1,0], [1,0,0,1] &[0,0,1,1], [0,1,0,0], [1,0,2,0], [1,1,0,1], [2,0,0,0] & [0,0,3,0], [0,1,1,1], [1,0,0,2], [1,0,1,0], [2,0,1,1], [2,1,0,0] & [0,0,1,2], [0,0,2,0], [0,1,0,1], [1,0,2,1], [1,1,0,2], [1,1,1,0], [2,0,0,1], [3,0,1,0]\\
		\hline
		i=4 & [0,1,0,1], [1,0,0,0] & [0,0,1,0], [1,0,1,1], [1,1,0,0] & [0,0,2,1], [0,1,1,0], [1,0,0,1], [2,0,0,2], [2,0,1,0] & [0,0,1,1], [1,0,1,2], [1,0,2,0], [1,1,0,1], [2,0,0,0], [3,0,0,1]\\
		\hline
	\end{tabular}
	\begin{proof}[Proof of Theorem \ref{raff}.] We will outline the various steps in a schematic way. The ideas are the same we have already discussed.\\
		\textbf{Degree 13:}\\
		Let $v\in sing \, \mfT^{13}(V)$. We first consider the term of $v$ in degree $(2|9)$; it can be different from zero only if $V$ appears in the decomposition of $\bigwedge^{9}(\mfs^\ast)\otimes V(\omega_1)\cong V([0,0,1,0])\oplus V([1,1,0,0])$. So if $V$ is not isomorphic to one of these two representations, the term in $(2|9)$ of $v$ must be equal to $0$, we look at next term, which is also $0$ because has degree $(3|7)$. Iterating, we deduce that $v=0$.\\
		If $V\cong V([0,0,1,0])$ or $V([1,1,0,0])$, take $\xi\in\mfs$ and consider the term in degree $(2|10)$ of $\xi v$. It can be non zero only if $V$ does not appear in $\bigwedge^{10}(\mfs^\ast)\otimes V(\omega_1) \cong V(\omega_1)$. It follows that it must be $0$ and, iterating, so does $\xi v$  $\forall \xi \in \mfs$ which in turn implies, as we already saw, that $v$ must be $0$. In conclusion, $sing \, \mfT^{13}(V)$ is always $0$.\\
		\textbf{Degree 12:}\\
		Let $v\in sing \, \mfT^{12}(V)$. We consider the term of degree $(1|10)$; it can be non zero only if $V$ appears in $\bigwedge^{10}(\mfs^\ast)\otimes V(\omega_i)\cong\omega_i$ for $i=0,\dotsc, 4$. Therefore if $V$ is not of the form $V(\omega_i)$, we check the term of degree $(2|8)$ which must be $0$ unless a copy of $V$ appears in $\bigwedge^8(\mfs^\ast)\otimes V(\omega_1)= V([0,1,1,0])\oplus V([1,0,0,1]) \oplus V([2,0,1,0])$. In the remaining cases $v=0$.\\
		Now assume $V\cong V([0,1,1,0])$, $V([1,0,0,1])$, $V([2,0,1,0])$ or $V(\omega_i)$ with $i=0,\dotsc, 4$.\\
		Take $\xi\in\mfs$ and consider the term of $\xi v$ of degree $(2|9)$. It cannot be non zero if $V$ does not appear in $\bigwedge^{9}(\mfs^\ast)= V([0,0,1,0])\oplus V([1,1,0,0])$. So if $V\ncong V(\omega_3)$ $\xi v=0$ $\forall \xi \in \mfs$ and again it implies that $v=0$. Therefore the only case in which we cannot rule out the presence of singular vectors of degree $12$ is for $V\cong V(\omega_3)$.\\
		\textbf{Degree 11:}
		Let $v\in sing \, \mfT^{11}(V)$. We consider the term of degree $(1 | 9)$; it can be different from zero when $V$ appears in $\bigwedge^{9}(\mfs^\ast)\otimes V(\omega_i)$ for $i=0,\dotsc, 4$. It happens when $V$ has as highest weight one belonging to the first column of the table.\\
		If we assume $V$ is not isomorphic to any of them, we can move to the next term which has degree $(2|7)$. According to the table, if $V$ is also not isomorphic to $V([0,0,1,1])$, $V([1,0,2,0])$, $V([1,1,0,1])$, $V([2,0,0,0])$ and $V([3,0,0,1])$, $v=0$.\\
		Assume $V$ is isomorphic to one of these representations, a fundamental representation or the trivial one; take $\xi\in\mfs$ and check the term of degree $(1|10)$; if $V$ is not a fundamental representation or the trivial one, we can move to the term of degree $(2|8)$ which will be $0$ unless a copy of $V$ appears in $\bigwedge^8(\mfs^\ast)\otimes V(\omega_1) \cong V([0,1,1,0])\oplus V([1,0,0,1])\oplus V([2,0,1,0])$. Therefore if $V$ is not isomorphic to $V([0,0,0,1])$, $V([0,0,1,0])$, $V([0,1,0,0])$, $V([1,0,0,0])$, $V([0,1,1,0])$ or $V([1,0,0,1])$, $sing \, \mfT^{11}(V)=0$.\\
		\textbf{Degree 10:}\\
		Let $v\in sing \, \mfT^{10}$. The term of $v$ with the greatest odd degree is the one of degree $(0|10)$. In this case we cannot deduce anything, since this term is constant from the point of view of $S(5)$, therefore singular.\\
		We relay again on Lemma \ref{easylemma}: take $\xi\in \mfs$ and consider $\xi v$, which has now degree $11$. This time we can still look at $\overline{ \xi v }\in \Gamma_{9}/(V)\Gamma_{8}(V)\cong T(\bigwedge^9(\mfs)\otimes V))$ which has even degree $1$. Again, we know that it is $0$ unless $V$ appears as an irreducible of $\bigwedge^9(\mfs^\ast)\otimes V(\omega_i)$, $i=0\dotsc, 5$. So if the highest weight of $V$ does not appear in the first column of the table, $\overline{ \xi v }=0$. We can then move to the next term, which has degree $(2|7)$. Here we look at the irreducible modules in $\bigwedge^7(\mfs^\ast)\otimes V(\omega_1)\cong V([0,0,1,1])\oplus V([1,0,2,0])\oplus V([1,1,0,1])\oplus V([2,0,0,0])\oplus V([3,0,0,1])$. So if in addition we ask that $V$ is not isomorphic to these modules, $\xi v=0$ $\forall \xi\in\mfs$ and again we cannot have singular vectors of degree $10$ in $\mfT(V)$.
	\end{proof}

\end{document}